\documentclass[a4paper]{amsart}

\usepackage{amssymb,comment}
\usepackage{amsfonts}
\usepackage{amsmath}
\usepackage{amsthm,graphicx}

\newcommand{\p}{{\mathbb P}}
\newcommand{\e}{{\mathbb E}}
\newcommand{\D}{{\mathrm d}}

\newcommand{\eqd}{\stackrel{\rm d}{=}}

\newcommand{\R}{\mathbb R}

\newcommand{\ii}{{\mathrm i}}

\newtheorem{thm}{Theorem}
\newtheorem{lemma}[thm]{Lemma}

\theoremstyle{remark}
\newtheorem{rem}{Remark}[section]

\theoremstyle{definition}

\begin{document}
		\title[Bivariate risk and queueing models with mutual assistance]{Stochastic decompositions in bivariate risk and queueing models with mutual assistance} 
		\author[J. Ivanovs]{Jevgenijs Ivanovs} 
		\address{Aarhus University}
\maketitle

\begin{abstract}
	We consider two bivariate models with two-way interactions in context of risk and queueing theory.
	The two entities interact with each other by providing assistance but otherwise evolve independently.
	We focus on certain random quantities underlying the joint survival probability and the joint stationary workload, and show that these admit stochastic decomposition. Each one can be seen as an independent sum of respective quantities for the two models with one-way interaction.
    Additionally, we discuss a rather general method of establishing decompositions from a given kernel equation by identifying two independent random variables from their difference, which may be useful for other models. Finally, we point out that the same decomposition is true for uncorrelated Brownian motion reflected to stay in an orthant, and it concerns the face measures appearing in the basic adjoint relationship. 
\end{abstract}

\keywords{Keywords: coupled processor, fluid network, reflection in orthant, kernel equation, basic adjoint relationship, stochastic decomposition, two-way interaction}

\section{Introduction}
Stochastic decomposition results abound in inventory management and queueing theory~\cite{polling,fuhrmann_cooper,ivanovs_kella}, but most of them concern models with vacations (switchover times) or a secondary jump input.
Here we present stochastic decomposition results of a different nature. 
We consider a bivariate model with two-way interaction and show that a certain fundamental quantity can be decomposed into two independent parts corresponding to models with one-way interactions. In fact, we do so in two frameworks: (i) a risk model of~\cite{boxma_ivanovs_risk} where each company covers the deficit of another and (ii) a queueing model with mutual assistance which can be seen as a coupled processor or a fluid network with two nodes.

Multivariate risk and queueing models are notoriously hard to analyze~\cite{cohen_boxma}. 
A classical example illustrating various difficulties is that of a Brownian motion reflected to stay in an orthant~\cite{dai_harrison,harrison_williams}.
In fact, our queueing model is a close relative. In Remark~\ref{rem:bm} we point out that our stochastic decomposition concerns the face measures appearing in the basic adjoint relationship. 
The focus here is not on defining complex models and establishing their properties which is known to be highly technical in general~\cite{harrison_williams,reed_zwart}, but rather on discovering structural results for some basic well-understood models. 
Thus we exclusively consider the case where the drivers are independent compound Poisson processes, 
so that without interactions we have (i) a pair of Cram\'er-Lundberg risk processes and (ii) a pair of M/G/1 workload processes. 

The quantities of interest are the joint survival probability in (i) and the joint stationary workload in (ii). 
Our results are neater and easier to interpret in the former case and thus we mainly focus on the risk model. 
Furthermore, we assume that the second company starts with 0 capital, because the general case can be reduced in some sense to such boundary cases. 
Then the minimal initial capital of the first company leading to joint survival is the quantity admitting stochastic decomposition.
This result may be useful when numerically evaluating survival probabilities for different strengths of interactions.
In queueing model we consider the stationary workload in the first queue given that the second queue is empty. In fact, decomposition holds true for a slightly different distribution with increased point mass at~0.

The models are defined in \S\ref{sec:models} and the main results are stated in \S\ref{sec:results}.
\S\ref{sec:proofs} contains proofs based on the results of~\cite{boxma_ivanovs_risk} and~\cite{boxma_ivanovs_queue} for the risk and queueing model, respectively. In \S\ref{sec:direct} we attempt to establish our decompositions directly from the kernel equation, 
which requires identification of two independent random variables from their difference, see \S\ref{sec:characterization}.
This latter approach may be useful in finding stochastic decompositions in other models.

\section{The models}\label{sec:models}
Throughout this work we assume that $X_1(t)$ and $X_2(t)$ are two independent drifted compound Poisson processes of the form 
\[X_i(t)=c_i t-\sum_{k=1}^{N_i(t)} J_{i,k},\qquad t\geq 0,\]
where $c_i>0$ and $J_{i,k},k=1,2,\ldots$ are positive iid random variables independent of the Poisson process $N_i(t)$.
The respective means are denoted by $\mu_i=\e X_i(1)$ and the Laplace exponents are given by
\[\psi_i(s)=\log \e e^{sX_i(1)}=c_i s+\lambda_i(\e e^{-s J_{i,1}}-1),\]
where $\lambda_i$ is the rate of the Poisson process $N_i(t)$.
Note that $X_i(t)$ started in $x_i\geq 0$ is the classic Cram\'er-Lundberg model in ruin theory, whereas $-X_i(t)$ reflected at~$0$ is the workload process in M/G/1 queue~\cite{ruin}.

 Our bivariate coupled risk and queueing models are defined below in an iterative way using the independent processes $\pm X_i(t)$ as drivers of the two entities, whereas the constants $r_1,r_2\in[0,\infty]$ parameterize certain interaction between the two. 
 The interaction is of the type where one company/server helps the other and vice versa.
Our focus is on the joint survival probability in risk model and the joint stationary workload in queueing model.
It is noted that unlike classical models, where survival probability and stationary workload are closely related by time-reversal argument~\cite{ruin}, we have no simple duality between the two quantities of interest. Nevertheless, some structural similarities on the level of kernel equations exist, which motivated looking at both models simultaneously.

\subsection{Coupled risk processes}
Let $x_i\geq 0$ be the initial capital of the company~$i$. It is assumed that the capitals evolve according to $x_i+X_i(t)$ until the first time when at least one of these processes becomes negative.
Note that this happens because of a claim $J_{i,k}$ received by one of the companies, since $X_i$ can not jump at the same instant a.s.
Letting $(y_1,y_2)$ be the current state, we restart the bivariate process from
\[(x_1,x_2)=\begin{cases}(0,y_2+r_1y_1),&\text{if }y_1<0,\\ (y_1+r_2y_2,0),&\text{if }y_2<0,\end{cases}\]
unless $x_1<0$ or $x_2<0$, in which case the ruin is declared. 

In words, \emph{deficit of the company $i$ is instantaneously covered by the other company which pays $r_i$ for the unit of capital transferred}. 
A standard scenario assumes that $r_1,r_2>1$ (think of taxation or transaction costs), whereas we allow for arbitrary rates in $[0,\infty]$. 
In particular, $r_1=0$ means that the first company refills to 0 without participation of the second, and $r_1=\infty$ means that deficit in the first company causes ruin in our bivariate model. Hence the boundary values $0,\infty$ yield a simpler model with one-way interaction.
This is not to say that latter models are easy to analyze.

Finally, we write $\phi(x_1,x_2)$ for the probability of survival (no ruin) from the initial capitals $(x_1,x_2)$ on the infinite time interval $[0,\infty)$. That is, the probability that companies manage to save each other at all times. Consider the bivariate transform 
\begin{equation}\label{eq:F}F(s_1,s_2)=\iint_{\mathbb R^2_+} e^{-s_1 x_1-s_2 x_2}\phi(x_1,x_2)\D x_1\D x_2, \qquad s_1,s_2\geq 0\end{equation}
Throughout this work we assume that our risk model satisfies the natural safety loading assumption:
\begin{equation}\label{eq:stability}\mu_1,\mu_2> 0\quad \text{or}\quad\mu_1\leq 0,\mu_2+r_1\mu_1>0\quad \text{or}\quad\mu_2\leq 0,\mu_1+r_2\mu_2>0,\end{equation}
and that $\mu_i>0$ if $r_i=\infty$, see Figure~\ref{fig:domain}.
\begin{figure}[h!]\label{fig:domain}
\includegraphics[width=0.6\textwidth]{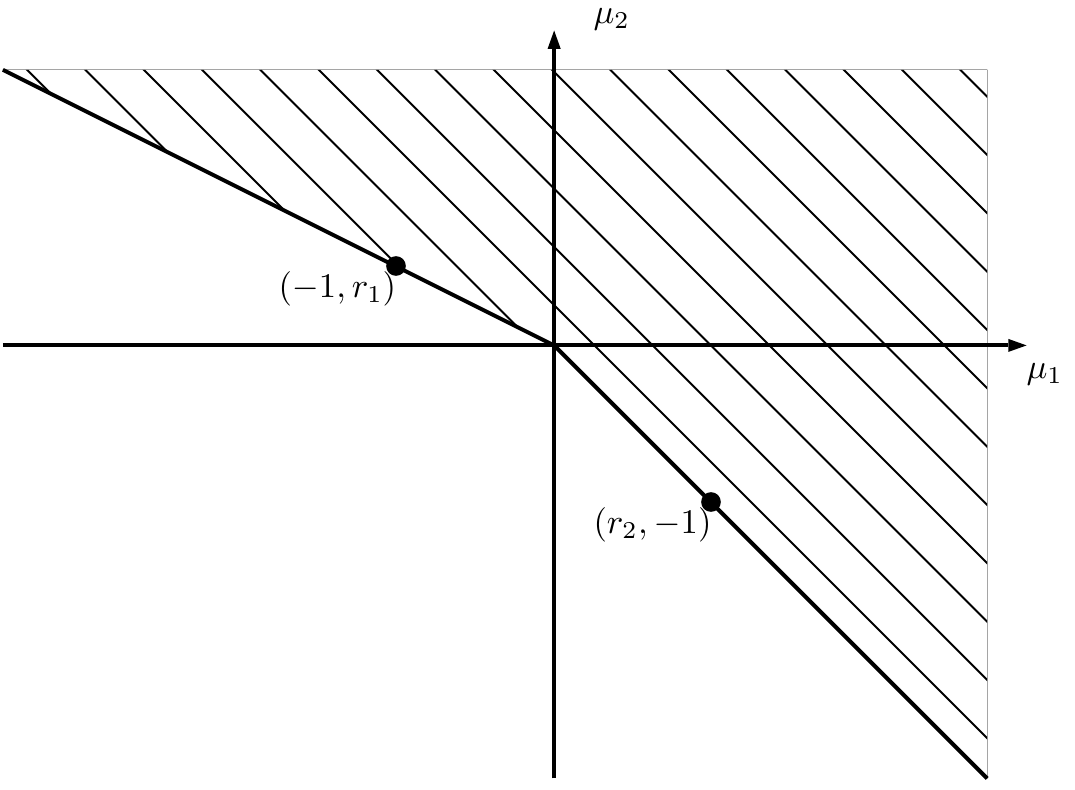}
\caption{Assumed parameter region. The case $r_1r_2=1$ corresponds to a linear boundary.}
\end{figure}
This guarantees that $\phi(x_1,\infty)=\phi(\infty,x_2)=1$ for $r_1,r_2\in [0,\infty)$, see~\cite{boxma_ivanovs_risk}; for $r_1=\infty$ we have $\phi(x_1,\infty)<1$ for any $x_1\in[0,\infty)$.

Importantly, there is a certain invariance under rescaling: for any $c>0$ the survival probability $\phi(x_1,x_2)$ in the original risk model is the same as the survival probability $\phi(x_1,cx_2)$ in the model $(X_1,cX_2,cr_1,r_2/c)$. In particular, the case of $r_1r_2=1$ can be reduced to the trivial case $r_1=r_2=1$, which is essentially one-dimensional: 
the risk problem reduces to the survival probability of the sum process.

\subsection{Coupled queueing processes}
We assume that workload processes evolve according to $Y_i(t)=-X_i(t)$ until at least one of these processes hits~0 (no work in the corresponding queue). That is, server $i$ works at speed $c_i$ and experiences customers bringing $J_{i,k}$ amount of work. After hitting the 0 the $i$-th workload process stays at~0 until arrival of the next customer to server~$i$, but the other server $j\neq i$ works at speed $c_j+c_i\rho_i$ during this time, where $\rho_i\in[0,\infty)$. One may view this model as a two-dimensional Skorokhod's reflection (see~\cite{kella_reflecting} and references therein), but only when $\rho_1\rho_2<1$, because otherwise the local times can not cancel each other when both queues are empty. 
In words, \emph{the server $i$, when idle, helps the other by providing 'proportion'~$\rho_i$ of his own service rate~$c_i$}. 

Our main quantity of interest is the pair of stationary workloads $(W_1,W_2)$, when it exists, and its bivariate transform
\begin{equation}\label{eq:G}G(s_1,s_2)=\e e^{-s_1 W_1-s_2 W_2}, \qquad s_1,s_2\geq 0.\end{equation}
The stability region is given by~\cite{cohen_boxma}
\begin{equation}\label{eq:stability_queue}\mu_1,\mu_2> 0\quad \text{or}\quad\mu_1\leq 0,\mu_1+\rho_2\mu_2>0\quad \text{or}\quad\mu_2\leq 0,\mu_2+\rho_1\mu_1>0,\end{equation}
which corresponds to the parameter region of the risk model by letting $r_i=1/\rho_j$ for $i\neq j$.
Note that the regime $\rho_1\rho_2<1$ is of the main interest and, moreover, it corresponds to a two-dimensional fluid network with an arbitrary legal routing matrix~\cite{boxma_ivanovs_queue}. 

Finally observe that the stationary workload $(W_1,W_2)$ in the original queueing model has the same law as $(W_1,W_2/c)$ in the model $(X_1,cX_2,c\rho_1,\rho_2/c)$.
Unlike the risk model, the queueing model with $\rho_1\rho_2=1$ is not trivial, see~\cite{cohen_boxma} for an in-depth study of this case.

\section{Behavior at the boundary and stochastic decompositions}\label{sec:results}
The first step in the analysis of the transforms~\eqref{eq:F} and~\eqref{eq:G} is to derive the corresponding so-called kernel equations, which identify the bivariate transform of interest in terms of two univariate functions relating to the behavior of the system at the boundaries. 
These derivations while being tedious follow some standard reasoning: application of the infinitesimal generator to $\phi(u,v)$ in risk~\cite{boxma_ivanovs_risk}, and level-crossing~\cite{cohen_boxma} or martingale arguments~\cite{boxma_ivanovs_queue} in queueing.
A closely related basic adjoint relationship for the stationary distribution of Brownian motion reflected to stay in an orthant can be found in~\cite{dai_harrison,harrison_williams}; it is derived using tools from stochastic calculus.

Additional motivation is provided by the following perspective in case of the risk model.
Note that processes evolve independently and without knowledge of $r_i$ until one of the companies gets in trouble.
The latter is then restarted from 0 at the cost of the other company, and so the system is at the boundary at this instant.

\subsection{Kernel equations}\label{sec:kernel}
The risk equation reads (with $s_1,s_2\geq 0$) 
\begin{equation}\label{eq:kernel_risk}(\psi_1(s_1)+\psi_2(s_2))F(s_1,s_2)=\frac{\psi_2(s_2)-\psi_2(r_2s_1)}{s_2-r_2 s_1}F_1(s_1)+
\frac{\psi_1(s_1)-\psi_1(r_1s_2)}{s_1-r_1 s_2}F_2(s_2),\end{equation}
where \[F_1(s)=\int_0^\infty e^{-sx}\phi(x,0)\D x.\]
In the case of $r_1=\infty$ (or $r_2=\infty$) the kernel equation should be read in the limiting sense, where $\psi_i(\theta)\sim c_i\theta$ as $\theta\to\infty$.

The queueing equation for $\rho_1\rho_2\neq 1$ is
\begin{equation}\label{eq:kernel_queue}(\psi_1(s_1)+\psi_2(s_2))\e e^{-s_1 W_1-s_2W_2}
=(s_2-\rho_2s_1)G_1(s_1)+(s_1-\rho_1s_2)G_2(s_2),\end{equation}
where 
\[G_1(s)=c_2\e (e^{-s W_1};W_2=0)+\rho_1\frac{c_2\rho_2+c_1}{1-\rho_1\rho_2}\p(W_1=W_2=0).\] 
Note that by joining the two terms containing $\p(W_1=W_2=0)$ we may also state an equation for the case $\rho_1\rho_2=1$, but then the structure of that kernel equation is different and it does require another type of analysis, see~\cite{cohen_boxma}.

Importantly, the original problem reduces to the problem of identification of the boundary functions $F_i(s)$ and $G_i(s)$.
In \cite{boxma_ivanovs_queue} and \cite{boxma_ivanovs_risk} these functions where expressed through the Wiener-Hopf factors of some auxiliary two-sided L\'evy process, see Section~\ref{sec:WH} for the summary of the results. Here we establish certain stochastic decomposition results underlying these unknown functions, which is the main result of this work.
We exclusively focus on $F_1(s)$ and $G_1(s)$, since the treatment of the other functions is analogous.

\subsection{Stochastic decompositions in risk}\label{sec:dec_risk}
Integration by parts yields the following identity
\[\widehat F_1(s)=s F_1(s)=\int_{0-}^\infty e^{-s u}\D\phi(u,0)=\e e^{-s U},\]
where $U$ is a non-negative random variable with c.d.f.\ $\p(U\leq u)=\phi(u,0)$ for $u\geq 0$. 
Note that, indeed, $\phi(u,0),u\geq 0$ is an increasing function taking values in $[0,1]$. Importantly, $\phi(\infty,0)=1$ according to our assumptions and so the random variable $U$ is proper, unless $r_2=\infty$. Finally, the distribution $\phi(u,0)$ has an atom at~0, i.e.\ $\phi(0,0)>0$, which explains the left integration limit $0-$ above.

We define $U$ as a functional of the two sample paths $(X_1(t),X_2(t)),t\geq 0$: it is the minimal initial capital of the first company leading to survival when the second company starts at 0 (such a minimum is always achieved). Indeed, $\p(U\leq u)$ is the probability that there is survival in our model for the initial capitals $(u,0)$.
In the following we write $U_{r_1,r_2}$ to make the dependence on $r_i$ explicit, and note that all such random variables are defined on the original probability space.
Observe that by construction $U_{r_1,r_2}$ is non-decreasing in both $r_1$ and $r_2$.

A useful perspective is provided by Figure~\ref{fig:perspective}, where we depict the second path upside down and shift the ground level according to the model specification. In other words, the grey regions are scaled according to $r_1$ and $r_2$. 
One can think of starting the first company with very large initial capital and then reducing it until the paths touch, so that no further decrease is possible. This procedure yields~$U$ as the difference between the final starting points. Observe the complexity of the model: the sample paths of the resultant processes (as in the picture) may change dramatically with a change of the initial capital of the first company alone.
\begin{figure}[h!]\label{fig:perspective}
\includegraphics[width=0.4\textwidth]{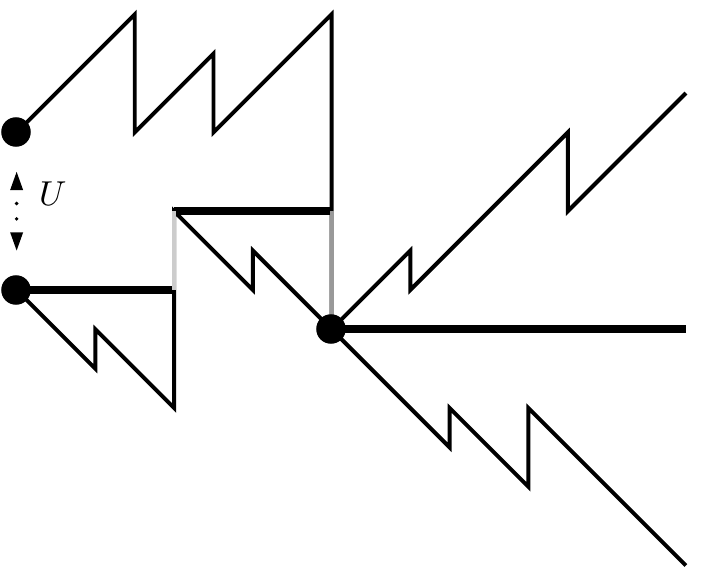}
\caption{Useful perspective: ground level is in bold, the rescaled regions are in grey.}
\end{figure}

Finally, we let $\underline X_i=\inf\{X_i(t):t\geq 0\}$ which is proper when $\mu_i>0$, and note that 
\[U_{\infty,0}=-\underline X_1.\] 
Moreover, $-\underline X_i$ has the distribution of the stationary workload in a single queue driven by $Y_i=-X_i$ process.
In particular, its transform is given by the generalized Pollaczek-Khinchine formula: $\e e^{\theta \underline X_i}=\mu_i\theta/\psi_i(\theta), \theta\geq 0.$

Let us now formulate our main result reducing the original problem with two-way interactions to two simpler problems with one-way interactions.
\begin{thm}\label{thm:dec_risk}
For $r_1,r_2\in(0,\infty)$ satisfying~\eqref{eq:stability} it holds that
\begin{equation}\label{eq:dec_main}U_{r_1,r_2}\eqd U_{r_1,0}+U'_{0,r_2},\end{equation}
where $U'$ denotes an independent copy of $U$. Moreover, we also have
\begin{align}
\label{eq:supp1} U_{r_1,0}&\eqd U_{r_1,\infty}|\{U_{0,\infty}=0\},&\mu_2>0,\\
\label{eq:supp2} U_{\infty,r_2}&\eqd U_{0,r_2}+U'_{\infty,0},&\mu_1>0.
\end{align}
\end{thm}
This result will be proven later in \S\ref{sec:proofs}.
Let us briefly comment on the term $U_{r_1,\infty}|\{U_{0,\infty}=0\}$ in the case of $\mu_2>0$. As mentioned before, the random variable $U_{r_1,\infty}$ is not proper, but it does become such upon conditioning on $X_2$ never becoming negative.
Note that using~\eqref{eq:supp1} and~\eqref{eq:supp2} we may provide further decompositions. In particular, \eqref{eq:supp2} leads to
\begin{equation}U_{r_1,r_2}+U'_{\infty,0}\eqd U_{r_1,0}+U'_{\infty,r_2},\qquad \mu_1>0.\label{eq:supp} \end{equation}

Finally, observe that~\eqref{eq:dec_main} can be restated in terms of a convolution equation for survival probabilities
\[\phi_{r_1,r_2}(u,0)=\int_{0-}^u\phi_{0,r_2}(u-x,0)\D \phi_{r_1,0}(x,0)=\int_{0-}^u\phi_{r_1,0}(u-x,0)\D \phi_{0,r_2}(x,0).\]

\subsection{Law invariance in risk}\label{sec:law_inv}
Importantly, there is another decomposition
\begin{equation}\label{eq:dec_alt}U_{r_1,r_2}\stackrel{d}{=}U'_{0,r_2}+U_{r_1,r_2}| \{U_{0,r_2}=0\},\end{equation}
which is not of the type presented in Theorem~\ref{thm:dec_risk} as it contains $U_{r_1,r_2}$ on both sides. 
The proof is essentially given in Figure~\ref{fig:paths}. 
\begin{figure}[h!]\label{fig:paths}
\includegraphics[width=0.8\textwidth]{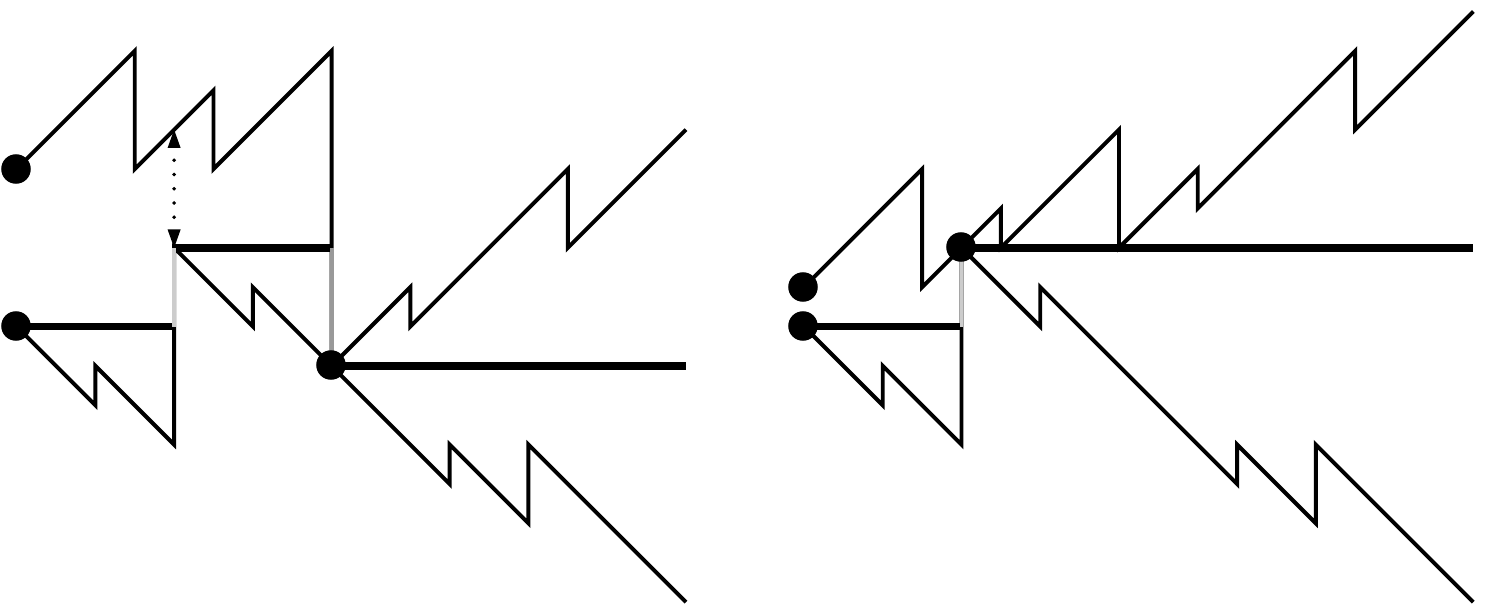}
\caption{Decomposition of another type}
\end{figure}
Consider $U'_{0,r_2}$ (as in the right picture) and find the first time $T$ where the paths touch. It must be that both companies are at 0 at~$T$. The quantity $U'_{0,r_2}$ depends only on the evolution of paths up to time $T$ and the fact that post-$T$ paths yield $U_{0,r_2}=0$. Now splitting of the paths at $T$ yields the decomposition in~\eqref{eq:dec_alt}. 
Furthermore, this decomposition also holds for two-sided processes $X_i$, i.e., when jumps of both signs are present.
Finally, note that one can not interchange the roles of $r_i$ above.

Comparing~\eqref{eq:dec_alt} and~\eqref{eq:dec_main} we find that 
\begin{equation}\label{eq:lawinv}U_{r_1,r_2}| \{U_{0,r_2}=0\}\eqd U_{r_1,0}.\end{equation}
It is noted that $X+Y\eqd X+Z$ for independent real random variables does not imply $Y\eqd Z$~\cite[p.\ 479]{feller}. 
This implication is true, however, for non-negative random variables, since the Laplace transform $\e e^{-s X},s\geq 0$ can be 0 only at isolated points.
Identity~\eqref{eq:lawinv} states that $U_{r_1,r_2}| \{U_{0,r_2}=0\}$ is law-invariant in $r_2$. This also extends to the boundary cases: for $r_2=0$ we simply obtain $U_{r_1,0}$, and for $r_2=\infty$ (when $\mu_2>0$) we get $U_{r_1,\infty}| \{U_{0,\infty}=0\}$ which  has  the distribution of $U_{r_1,0}$ according to~\eqref{eq:supp1}.


\subsection{Stochastic decomposition in queueing}
Here we assume that $\rho_1\rho_2<1$ and that $\mu_1,\mu_2>0$.
Similarly to the risk problem, we need to focus on a random quantity underlying $G_1$ in some sense. 
Note that the kernel equation implies $\mu_2=G_1(0)-\rho_1G_2(0)$ and $\mu_1=-\rho_2G_1(0)+G_2(0)$; to see this let $s_i=0$, divide both sides by $s_j$ and let it decrease to 0.
Hence $G_1(0)=(\mu_2+\rho_1\mu_1)/(1-\rho_1\rho_2)$ and thus
\begin{align}\nonumber &\widehat G_1(s):=G_1(s)/G_1(0)\\
&=\frac{c_2(1-\rho_1\rho_2)}{(\mu_2+\rho_1\mu_1)}\e (e^{-s W_1};W_2=0)+\rho_1\frac{c_2\rho_2+c_1}{\mu_2+\rho_1\mu_1}\p(W_1=W_2=0).\label{eq:transform}\end{align}
Since under our assumptions both terms are positive, we see that $\widehat G_1(s)=\e e^{-s V_{\rho_1,\rho_2}}$ is the transform of the mixture: $W_1|\{W_2=0\}$ and $0$ with obvious probabilities. Note that we put additional mass at 0 as compared to the distribution of $W_1|\{W_2=0\}$.

\begin{thm}\label{thm:dec_queue}
For $\rho_1\rho_2<1$ and $\mu_1,\mu_2>0$ it holds that 
\[V_{\rho_1,\rho_2}+V'_{0,0} \eqd V_{\rho_1,0}+V'_{0,\rho_2},\]
where $V_{\rho_1,\rho_2}$ is defined by the transform in~\eqref{eq:transform} and, in particular, $V_{0,0}$ has the stationary distribution of the stand-alone first queue.
\end{thm}
Let us comment on the assumed conditions. 
Stability of the system with rate pair $(\rho_1,0)$ implies that $\mu_1>0$, and similarly we get $\mu_2>0$, and so the assumption $\mu_1,\mu_2>0$ is necessary. 
Furthermore, for $\rho_1\rho_2>1$ the first term in~\eqref{eq:transform} becomes negative implying that $\widehat  G_1(s)$ can not be a transform of a random variable. 
Of course, there may exist decompositions for some other random quantity, which also works for $\rho_1\rho_2>1$, but we were not able to identify such.

\begin{rem}\label{rem:bm}
	Exactly the same decomposition as in Theorem~\ref{thm:dec_queue} is also true for the Brownian case, i.e., when $X_1$ and $X_2$ are two independent linear Brownian motions. The interpretation of $V$ is, however, different. Note that $\p(W_i=0)=0$ in this case.
	According to~\cite{boxma_ivanovs_queue} we have
	\[G_1(s)=\e^*\int_0^1 e^{-sW_1(t)}\D L_2(t),\]
	where $L_2$ is the regulator at~0 and $\e^*$ signifies that the system is started in stationarity. Hence we see that $G_1(s)/G_1(0)$ is a transform of a positive random variable whose law is given by
	\[\p(V\in B)=\e^*\int_0^1 1_B(W_1(t))\D L_2(t)/\e^*L_2(1).\]
	Note that this is exactly the face measure $\nu_2$ in the basic adjoint relationship, see e.g.~\cite[(6)]{dai_harrison}, rescaled to be a probability measure.
\end{rem}

\section{Proofs via reduction to Wiener-Hopf factors}\label{sec:proofs}
Proofs given in this section are based on the expressions of $F_1$ and $G_1$ provided in~\cite{boxma_ivanovs_queue} and~\cite{boxma_ivanovs_risk}, which are  in terms of the Wiener-Hopf factors of some auxiliary two-sided L\'evy process.
We write $x^\pm=\max(\pm x,0)$ for the positive/negative part of~$x$.

\subsection{Auxiliary process and its Wiener-Hopf factors}\label{sec:WH}
Let $\Phi_i(\theta),\theta>0$ be the unique positive inverse of $\psi_i$, which can be analytically continued to all complex $\theta$ with non-negative real part.
It is noted that $\Phi_i(0)$ is 0 or strictly positive according to $\mu_i\geq 0$ and $\mu_i<0$.
One of the main observations in~\cite{boxma_ivanovs_queue} is that for any $r\in(0,\infty)$
\[\psi_r(\theta)=-\frac{\theta}{\Phi_1(\theta)}+\frac{r\theta}{\Phi_2(-\theta)},\qquad \theta\in\ii\mathbb R\]
is the L\'evy exponent of some L\'evy process~$X_r(t)$, i.e., $\psi_r(\theta)=\log\e e^{-\theta X_r(1)}$.
This auxiliary process is killed at rate
\[k_r=-\psi_r(0)=\mu_1^++r\mu_2^+>0,\]
where the value at~0 is interpreted in the limiting sense.
Hence we may define the corresponding Wiener-Hopf factors 
\[\Psi_r^+(\theta)=\e e^{-\theta\sup_t\{X_r(t)\}},\qquad \Psi_r^-(\theta)=\e e^{-\theta\inf_t\{X_r(t)\}}\]
satisfying $\Psi^+_r(\theta)\Psi_r^-(\theta)=-k_r/\psi_r(\theta)$ for $\theta\in\ii\mathbb R$. In fact, in the following we will only need the first factor $\Psi^+_r(s)$ for $s>0$.

Let us remark that we can write $X_r(t)=Z_1(t)-Z_2(rt)$ with independent $Z_1$ and $Z_2$, where $Z_i$
is the drift-less compound Poisson process with positive jumps characterized by 
$\log \e e^{-\theta Z_i(1)}=-\theta/\Phi_i(\theta)$, and so it is killed at rate $\mu_i^+\geq 0$. In fact, $Z_i$ is the descending ladder time process of $X_i$, see~\cite[\S 6.5.2]{kyprianou}.

It is noted that in~\cite{boxma_ivanovs_queue} and~\cite{boxma_ivanovs_risk} two auxiliary L\'evy processes with slightly different representations where used.
Using one family of processes $X_r(t)$ instead made the final formulas cleaner and led to some further important observations.

\subsection{Expressions of the unknown functions}
The expressions for $G_i(s)$ and $F_i(s)$ appearing in the kernel equations~\eqref{eq:kernel_queue} and~\eqref{eq:kernel_risk} were identified in~\cite{boxma_ivanovs_queue} and \cite{boxma_ivanovs_risk}, respectively, using an educated-guess approach based on intuition from~\cite{cohen_boxma}, see also the latter work for an alternative expression of $G_i(s)$ based on the random walk theory. 
For the risk model with $r_i\in(0,\infty)$ and satisfying~\eqref{eq:stability} we have
  \begin{align}\label{eq:FF}
F_1(s)=\frac{\mu_1^+-r_1r_2\mu_1^-+r_2\mu_2}{\psi_1(s)+\psi_2(r_2s)}\frac{\Psi_{1/r_1}^+(\psi_1(s))}{\Psi_{r_2}^+(\psi_1(s))}, \qquad s>\Phi_1(0).
\end{align}
For the queueing model with $\rho_i\in(0,\infty)$, $\rho_1\rho_2\neq 1$ and satisfying~\eqref{eq:stability_queue} we have
 \begin{align}\label{eq:GG}
 G_1(s)=\frac{\mu_2+\rho_1\mu_1^++\mu_1^-/\rho_2}{1-\rho_1\rho_2}\frac{\Psi_{\rho_2}^+(\psi_1(s))}{\Psi_{1/\rho_1}^+(\psi_1(s))}, \qquad s>\Phi_1(0).
\end{align}
Note that in the case $\mu_1>0$ it must be that $\Phi_1(0)=0$ showing again that $G_1(0)=(\mu_2+\rho_1\mu_1)/(1-\rho_1\rho_2)$.

\subsection{Proofs}
Key observation is that the Wiener-Hopf factors defined in \S\ref{sec:WH} have a very simple limiting form when $r\uparrow \infty$ and $r\downarrow 0$.
\begin{lemma}For any $s\neq 0$ with $\Re(s)\geq 0$ it holds that
\begin{align}\label{eq:WH_limits}&\Psi^+_\infty(s)=1,&\Psi^+_0(s)=\mu^+_1\Phi_1(s)/s.\end{align}
Moreover, for $\mu_1\leq 0,\mu_2>0$ we also have 
\begin{equation}\label{eq:limit}\lim_{r\downarrow 0}\Psi^+_r(s)/r=\mu_2\Phi_1(s)/s.\end{equation}
\end{lemma}
\begin{proof}
If $\mu_2> 0$ then $k_r\to \infty$ as $r\to \infty$ implying that the killing epoch approaches~0. Moreover, the supremum of $X_r(t)$ is bounded by $Z_1$ at the killing time and so $\Psi^+_\infty(s)=1$. If $\mu_2\leq 0$ then the killing rate stays constant, whereas $Z_2(r\epsilon)\to\infty$ a.s. for any $\epsilon>0$ implying the same result.

As $r\downarrow 0$ the killing rate approaches $\mu_1^+$, whereas $X_r(t)\to Z_1(t)$ for every $t>0$. Thus for $\mu_1\leq 0$ the supremum tends to $\infty$ yielding $\Psi^+_0(s)=0$, and for $\mu_1>0$ the supremum becomes $Z_1$ evaluated at its killing time of rate $\mu_1$ and so 
\[\Psi^+_\infty(s)=\int_0^\infty e^{(-s/\Phi_i(s)+\mu_1)t}\mu_1e^{-\mu_1t}\D t=\mu_1\Phi_1(s)/s.\]

Finally, consider the case $\mu_1\leq 0$ and note that 
\[\overline X_r:=\sup_t\{Z_1(t)-Z_2(rt)\}=\sup_t\{Z_1(t/r)-Z_2(t)\},\] where $Z_1$ has infinite life-time and $Z_2$ is killed at an independent exponential time $e_{\mu_2}$ of rate $\mu_2>0$. Observe that $Z_1(e_{\mu_2}/r)-Z_2(e_{\mu_2})\leq \overline X_r\leq Z_1(e_{\mu_2}/r)$ and
\[\e e^{-sZ_1(e_{\mu_2}/r)}/r=\e e^{-s/\Phi_1(s)e_{\mu_2}/r}/r=\frac{\mu_2}{s/\Phi_1(s)+r\mu_2}\to \mu_2\Phi_1(s)/s\qquad \text{as }r\downarrow 0.\]
A similar calculation shows that the transform of the lower bound has the same limit, which concludes the proof.
\end{proof}

\begin{proof}[Proof of Theorem~\ref{thm:dec_risk}]
First, consider the case $\mu_1>0$.
Writing $\widehat F^{r_1,r_2}_1(s)$ to stress the dependence of the transform $\widehat F_1(s)$ on the rates $r_1,r_2$, we find from~\eqref{eq:FF} and~\eqref{eq:WH_limits} that 
\begin{align*}
\widehat F^{0,r_2}_1(s)&=\frac{\mu_1^++r_2\mu_2}{\psi_1(s)+\psi_2(r_2s)}\frac{s}{\Psi_{r_2}^+(\psi_1(s))}, &\widehat F^{r_1,0}_1(s)&=\Psi_{1/r_1}^+(\psi_1(s))
\end{align*}
for $s>\Phi_1(0)$, where the continuity of $\widehat F^{r_1,r_2}_1(s)$ at $r_1=0$ and at $r_2=0$ is straightforward to check. Thus we indeed have
\begin{equation}\label{eq:product}\widehat F^{0,r_2}_1(s) \widehat F^{r_1,0}_1(s)=\widehat F^{r_1,r_2}_1(s)\end{equation}
proving~\eqref{eq:dec_main}.
When $\mu_1\leq 0$ we get with the help of~\eqref{eq:limit}
\[\widehat F^{r_1,0}_1(s)=s\frac{r_1\mu_1+\mu_2}{\psi_1(s)}\Psi_{1/r_1}^+(\psi_1(s))\lim_{r_2\downarrow 0}\frac{r_2}{\Psi_{r_2}^+(\psi_1(s))}=\frac{r_1\mu_1+\mu_2}{\mu_2}\Psi_{1/r_1}^+(\psi_1(s)),\] whereas $\widehat F^{0,r_2}_1(s)$ has the same form as above. Thus~\eqref{eq:product} holds and implies the required decomposition also when $\mu_1\leq 0$.

The identity~\eqref{eq:supp1} has been already proven in \S\ref{sec:law_inv} assuming~\eqref{eq:dec_main}, 
but it also follows from
\[\widehat F^{r_1,\infty}_1(s)=\frac{-r_1\mu_1^-+\mu_2}{c_2}s\Psi_{1/r_1}^+(\psi_1(s))=\widehat F^{r_1,0}_1(s)\frac{\mu_2}{c_2},\]
because $\psi_2(s)/s\to c_2$ as $s\to\infty$. It is left to note that  $\widehat F^{r_1,\infty}_1(s)$ is the transform of the defective random variable $U^{r_1,\infty}$, where
the probability of $\{U^{r_1,\infty}\in [0,\infty)\}=\{\underline X_2=0\}$ is $\mu_2/c_2$.

Finally, with respect to~\eqref{eq:supp2} we have
\[\widehat F^{\infty,r_2}_1(s)=\widehat F^{0,r_2}_1(s)\frac{\mu_1 s}{\psi_1(s)}\]
and the result follows, because the latter factor is the transform of $-\underline X_1=U_{\infty,0}$ according to the generalized Pollaczek-Khinchine formula.
\end{proof}

\begin{proof}[Proof of Theorem~\ref{thm:dec_queue}]
Under the assumed conditions we have, according to~\eqref{eq:GG},
\[\widehat G^{\rho_1,\rho_2}_1(s)=\frac{\Psi_{\rho_2}^+(\psi_1(s))}{\Psi_{1/\rho_1}^+(\psi_1(s))}.\]
Using~\eqref{eq:WH_limits} we find that
\[\widehat G^{\rho_1,0}_1(s)=\frac{\mu_1 s}{\psi_1(s)}\frac{1}{\Psi_{1/\rho_1}^+(\psi_1(s))},\qquad \widehat G^{0,\rho_2}_1(s)=\Psi_{\rho_2}^+(\psi_1(s))\]
and the decomposition result follows by noting that $\mu_1 s/\psi_1(s)$ is the transform of the stationary workload in the first stand-alone queue (generalized Pollaczek-Khinchine formula).
\end{proof}
Note that we have exactly the same factorization in the case $\rho_1\rho_2>1$, but then the probabilistic interpretation of $\widehat G^{\rho_1,\rho_2}_1(s)$ is lost.
Even more interestingly, $\widehat G^{\rho_1,\rho_2}_1(s)=\widehat G^{\rho_1,\infty}_1(s)\widehat G^{0,\rho_2}_1(s)$, but the first term on the right can not be a transform of a random variable for any $\rho_1\in(0,\infty)$.

\section{Observing decompositions directly}\label{sec:direct}
In this section we provide an alternative, more direct, way to establish the decompositions stated in Theorem~\ref{thm:dec_risk} and Theorem~\ref{thm:dec_queue}, which avoids identification of the transforms $\widehat F_1(s),\widehat G_1(s)$ in terms of Wiener-Hopf factors.
This approach, however, relies on certain properties of the transforms  and requires some additional theory. It is presented here mainly because of its methodological interest. Moreover, it can be seen as a simple heuristic tool for discovering stochastic decompositions in other models.

\subsection{Distributional equality of differences}
Note that the kernel equations~\eqref{eq:kernel_risk} and~\eqref{eq:kernel_queue} have a somewhat similar form:
\[(\psi_1(s_1)+\psi_2(s_2))\times \cdot=A_1(s_1,s_2)T_1(s_1)+A_2(s_1,s_2)T_2(s_2),\]
where $T_i$ is the transform of interest $\widehat F_i$ or $\widehat G_i$, and $A_1(s_1,s_2)$ is
\[\frac{\psi_2(s_2)-\psi_2(r_2s_1)}{s_1(s_2-r_2 s_1)}\qquad\text{ or }\qquad(\mu_2+\rho_1\mu_1)(s_2-\rho_2s_1).\]
Proceeding as in~\cite{cohen_boxma,boxma_ivanovs_queue} we let $s_1=\Phi_1(\theta),s_2=\Phi_2(-\theta)$ with non-zero $\theta\in\ii\mathbb R$ to annihilate the left hand side and to obtain the identity:
\[A_1(s_1,s_2)T_1(s_1)=-A_2(s_1,s_2)T_2(s_2)\]
for the given choice of $s_1,s_2$ dependent on $\theta$.

From now on we exclusively focus on the risk problem, because the queueing problem can be treated in a very similar way.
Firstly, we have 
\[A^{0,r_2}_1(s_1,s_2)A^{r_1,0}_1(s_1,s_2)=A^{r_1,r_2}_1(s_1,s_2)\frac{\psi_2(s_2)}{s_1s_2}\]
Assuming that $s_2\neq r_2s_1$ and $s_1\neq r_1s_2$ we readily obtain the identity
\begin{equation}\label{eq:transform_product}T^{r_1,r_2}_1(s_1)\left(T^{0,r_2}_2(s_2)T^{r_1,0}_2(s_2)\right)=\left(T^{0,r_2}_1(s_1)T^{r_1,0}_1(s_1)\right)T^{r_1,r_2}_2(s_2),\end{equation}
because $\psi_2(s_2)=-\theta=-\psi_1(s_1)$ and the cancelled out factors must be non-zero. Importantly, \eqref{eq:transform_product} holds for all $\theta\in\ii\R$, which follows by continuity and the fact that the set of excluded $\theta$ must not contain an interval of $\ii\R$, see e.g.~\cite{boxma_ivanovs_queue} for the properties of $\Phi_i(\theta)$.

It is well known that $\e e^{-\theta \tau_i(x)}=e^{-\Phi_i(\theta)x}$, where $\tau_i(x)=\inf\{t\geq 0:X_i(t)>x\}$ is the first passage time of the $i$th process over $x\geq 0$.
Hence $T^{r_1,r_2}_1(\Phi_1(\theta))$ is the transform of a non-negative random variable $\tau_1(U'_{r_1,r_2})$, and similarly $T^{r_1,r_2}_2(\Phi_2(-\theta))$ is the transform of a non-positive random variable. Thus \eqref{eq:transform_product} translates into
\begin{equation}\label{eq:diff}Z_1-Z_2\eqd Z_1'-Z_2',\end{equation}
where $(Z_1,Z_2)$ and $(Z_1',Z_2')$ are two pairs of independent non-negative random variables. It is left to show that $Z_1\eqd Z_1'$ since then 
the decomposition~\eqref{eq:dec_main} follows immediately.
It is not true in general, however, that~\eqref{eq:diff} implies $Z_1\eqd Z_1'$, and we address this question in the following.

\subsection{Characterization of two independent non-negative random variables by their difference}\label{sec:characterization}
The noteworthy result of~\cite{kotlarski} (see also~\cite{kagan_linnik_rao}) states that the laws of three independent real random variables $Z_1,Z_2,Z_3$ are characterized up to a change of location by the joint law of two differences $(Z_1-Z_2,Z_2-Z_3)$, whenever the characteristic function of the latter pair does not vanish. This result extends to $n\geq 3$ independent random variables but fails for~$n=2$. The obvious counter-examples are $Z_1-Z_2=(c-Z_2)-(c-Z_1)$ and $Z_1-(Z_2'+Z)=(Z_1-Z)-Z'_2$ with some constant $c$ and an independent~$Z$.
Importantly, by imposing some assumptions on $Z_i$ which exclude the above types of examples, we may still obtain a characterization result.
\begin{lemma}\label{lemma:char}
Consider two independent non-negative random variables $Z_1,Z_2$ and for $i=1,2$ assume that the law of $Z_i$ has a density on $(0,\infty)$ and a point mass at~0, and that $\e e^{\theta Z_i}\neq 0$ for all $\theta$ with $\Re(\theta)\leq 0$.
Then the law of the difference $Z_1-Z_2$ characterizes the laws of $Z_i$ in the sense that 
\[Z_1-Z_2\eqd Z'_1-Z'_2\qquad \text{implies}\qquad Z_i\eqd Z'_i\] assuming that $Z_i'$ satisfy the above conditions.
\end{lemma}
\begin{proof}
Let us show that 
\begin{equation}\label{eq:limit2}\e e^{(\ii a-b) Z_j}\to p_j:=\p(Z_j=0)>0\qquad \text{as }|a|+b\to\infty\text{ with }b\geq 0.\end{equation} 
It is sufficient to prove that $\int_0^\infty e^{(\ii a-b) x} f(x)\D x\to 0$ for any density~$f$.
We may assume that $b\leq B$ and $|a|\to \infty$, because the case $b\to\infty$ is trivial.
As in the case of characteristic functions~\cite[Lem.\ XV.5.3]{feller}, it is enough to check that 
\[\int_u^v e^{(\pm\ii a-b)x}\D x=\frac{1}{a}\int_{ua}^{av} (\cos(x)\pm\ii\sin(x))e^{-xb/a}\D x\to 0\] for any $a>0$ and $v>u\geq 0$.
This follows from the fact that 
\[\max_{k\in[ua,va]}\left|\int_{k}^{k+2\pi} \cos(x)e^{-xb/a}\D x\right|\leq c(1-e^{-2\pi b/a})\to 0\]
and the same argument concerning the $\sin$ function.

Assuming that $Z_1-Z_2\eqd Z_1'-Z_2'$ and considering the transforms of both sides, we obtain
\begin{equation}\label{eq:imaginary}\e e^{\theta Z_1}/\e e^{\theta Z'_1}=\e e^{-\theta Z'_2}/\e e^{-\theta Z_2}\end{equation}
for purely imaginary~$\theta$.
We see that 
the right hand side is analytic in the right-half of the complex plane, and the left hand side is analytic in the left-half of the complex plane. Moreover, both sides are continuous and coincide on the imaginary axis. Hence one is analytic continuation of the other~\cite{lang}. According to~\eqref{eq:limit2} both sides are bounded in their respective half-planes, and so by Liouville's theorem~\cite{lang} this analytic function must be constant. By plugging in $\theta=0$ we find that the constant is $1$ and so $\e e^{\theta Z_i}=\e e^{\theta Z'_i}$ concluding the proof.

%

\end{proof}

It is known that the set of zeros of the characteristic function of a non-negative random variable, say $Z_1$, has 0 Lebesgue measure. Thus it is impossible to modify the other characteristic functions in~\eqref{eq:imaginary} on this set while preserving the continuity. 
In conclusion, the idea of constructing counterexamples from~\cite{kotlarski} would not work, and
 it seems very likely that the assumption of Lemma~\ref{lemma:char} that the transforms are non-vanishing is redundant. 

Instead of considering non-negative random variables with a point mass at 0 and a density on $(0,\infty)$, we may assume that the left-most point of the support is 0 and the right-most is $\infty$. This again rules out the above mentioned obvious counterexamples, and it would be important to understand if this is sufficient (under minor further assumptions, say) for the characterization of the laws of $Z_i$ by the law of their difference.

The main hurdle in applying Lemma~\ref{lemma:char} in the setting of~\eqref{eq:diff} is checking the (possibly redundant) assumption that the transforms are non-vanishing. This is indeed true when $\mu_1,\mu_2>0$, as can be seen from the reduction~\eqref{eq:FF} to the Wiener-Hopf factors. With regard to the other assumptions, we clearly have the point mass at 0, and then it is sufficient to check that $U$ has a density on $(0,\infty)$ which is the same as $\phi(u,0)$ being differentiable.
In conclusion, one may use the direct approach as a heuristic tool in establishing stochastic decompositions, whereas its rigorous application requires verification of various technical details as well as availability of characterization results similar to Lemma~\ref{lemma:char}.

\section{Concluding remarks}
The above presented approaches to proving stochastic decompositions are mainly analytic. Identities, as simple as~\eqref{eq:dec_main}, however, are asking for direct probabilistic arguments.
Discovering such proofs would be highly important in understanding these and possibly other coupled models.
The only such argument the author could find after a prolonged time is the one underlying~\eqref{eq:dec_alt}.
In this regard note that a probabilistic argument for the law invariance statement~\eqref{eq:lawinv} would be sufficient as it and~\eqref{eq:dec_alt} imply~\eqref{eq:dec_main}.
A related challenge concerns finding other multivariate models yielding stochastic decomposition results.


\section*{Acknowledgments}
This work is dedicated to Peter Taylor in appreciation of his contributions in applied probability.
The author is thankful to Stella Kapodistria and Offer Kella for their interest in the problem and insightful remarks. 
Support of Sapere Aude Starting grant is gratefully acknowledged.


\end{document}